\documentclass{amsart}
\usepackage{graphics}
\usepackage{ae}
\usepackage{amsmath,bbm}

\usepackage{amssymb}

\def\L{{\mathcal L}}

\def\1{{\mathbf 1}}

\numberwithin{equation}{section}

\newtheorem{theo}{Theorem}
\newtheorem{prop}[theo]{Proposition}

\newtheorem{coro}[theo]{Corollary}
\newtheorem{lemma}[theo]{Lemma}
\theoremstyle{definition}

\begin{document}

\title{Decidability of the HD$\mathbf{0}$L ultimate periodicity problem}
\author{Fabien Durand}
\address[F.D.]{\newline
Universit\'e de Picardie Jules Verne\newline
Laboratoire Ami\'enois de Math\'ematiques Fondamentales et
Appliqu\'ees\newline
CNRS-UMR 7352\newline
33 rue Saint Leu\newline
80039 Amiens Cedex 01\newline
France.}
\email{fabien.durand@u-picardie.fr}

\begin{abstract}
In this paper we prove the decidability of the HD$0$L ultimate periodicity problem.
\end{abstract}

\maketitle

\section{Introduction}

\subsection{The HD$\mathbf{0}$L ultimate periodicity problem}

In this paper we prove the decidability of the following problem :

\medskip

{\bf Input:}
Two finite alphabets $A$ and $B$, an endomorphism $\sigma : A^* \to A^*$, a word $w\in A^*$ and a morphism $\phi : A^* \to B^*$.

{\bf Question:}
Do there exist two words $u$ and $v$ in $B^*$, with $v$ non-empty, such that the sequence $(\phi (\sigma^n (w )))_n$ converges to $uv^\omega$ ({\em i.e.}, is ultimately periodic)?

\medskip

(The convergence of the sequence $(\phi (\sigma^n (w )))_n$ meaning that $(|\phi (\sigma^n (w ))|)_n$ goes to $+\infty$ and that $(\phi (\sigma^n (www\cdots )))_n$ converges in $B^\mathbb{N}$ endowed with the usual product topology.)
We will refer to it as the {\em HD$0$L ultimate periodicity problem}.
Observe that it is slightly more general than the classical statement where it is assumed in the input that the sequence $(\phi (\sigma^n (w )))_n$ converges.

\begin{theo}
\label{theo:main}
The HD0L ultimate periodicity problem is decidable.
\end{theo}

This result was announced in \cite{Durand:2012}.
While we were ending the writing of this paper, I. Mitrofanov put on Arxiv \cite{Mitrofanov:preprint2011} another solution of this problem. 

This problem was open for about 30 years.

In 1986, positive answers were given independently for D$0$L systems (or purely substitutive sequences) in both \cite{Harju&Linna:1986} and \cite{Pansiot:1986}, and, for automatic sequences (which are particular HD$0$L sequences) in \cite{Honkala:1986}.
Other proofs have been given for the D$0$L case in \cite{Honkala:2008} and for automatic sequences in \cite{Allouche&Rampersad&Shallit:2009}.

Recently in \cite{Durand:2012} the primitive case has been solved.

In \cite{Honkala&Rigo:2004} is given an equivalent statement of the HD$0$L ultimate periodicity problem in terms of recognizable sets of integers and abstract numeration systems.
In fact, J. Honkala already gave a positive answer to this question in \cite{Honkala:1986} but in the restricted case of the usual integer bases,  {\em i.e.}, for $k$-automatic sequences or
constant length substitutive sequences. Recently, in \cite{Bell&Charlier&Fraenkel&Rigo:2009}, a positive answer has been given for a (large) class of numeration systems including for instance the Fibonacci numeration system.

Let us point out that the characterization of recognizable sets of integers for abstract numeration systems in terms of substitutions given in \cite{Maes&Rigo:2002} (see also \cite{Lecomte&Rigo:2010}), together with Theorem \ref{theo:main}, provides a decision procedure to test whether a recognizable set of integers in some abstract numeration system is a finite union of arithmetic progressions.

\subsection{Organization of the paper}
\label{subsec:organization}
In Section \ref{section2} are the classical definitions.

In Section \ref{section3} we prove the HD$0$L ultimate periodicity problem for substitutive sequences.
These sequences are such that $(\sigma^n (w))_n$ converges.
This avoids to test the existence of the limit.
Indeed, there are examples where $(\phi (\sigma^n (w)))_n$ converges and $(\sigma^n (w))_n$ does not:
for $\sigma$ and $\phi$, defined by $\sigma (a) = cb$, $\sigma (b) =ba$, $\sigma (c) = ab$, $\phi (a) = \phi (c) = 0$ and $\phi (b) = 1$, the sequence $(\sigma^n (a))_n$  does not converge but $(\phi (\sigma^n (a)))_n$ does (to the Thue-Morse sequence).

Under these assumptions the proof could be sketched as follows. 
First we recall some primitivity arguments about matrices and substitutions.
The ''best or easiest situation'' is when we deal with growing substitutions and codings (letter-to-letter morphisms).
It is known that we can always consider we are working with codings (see \cite{Cobham:1968,Pansiot:1983,Allouche&Shallit:2003,Cassaigne&Nicolas:2003}).
In \cite{Honkala:2009} it is shown this can be algorithmically realized.
We propose a different algorithm using the proof of \cite{Cassaigne&Nicolas:2003} where we replace some (non-algorithmic) arguments (Lemma 2, Lemma 3 and Lemma 4 of this paper) by algorithmic ones.

We treat separately growing and non-growing substitutions.
For growing substitutions we look at their primitive components and we use the decidability result established in \cite{Durand:2012} about periodicity for primitive substitutions. Indeed, these primitive components should generate periodic sequences. Hence, we check it is the case (if not, then the sequence is not ultimately periodic). 
From there, Lemma \ref{lemma:finalcheck} allows us to conclude.

For the non-growing case we use a result of Pansiot \cite{Pansiot:1984} saying that we can either consider we are in the growing case or there are longer and longer periodic words with the same period in the sequence. 
We again conclude with Lemma \ref{lemma:finalcheck}.

In Section \ref{section4} we show how to use the substitutive case to solve the general HD$0$L case.
This concludes the proof of Theorem \ref{theo:main}.

\subsection{Questions and comments}
 
We did not compute the complexity of the algorithm provided by our proof of the HD$0$L ultimate periodicity problem. 
Looking at Proposition \ref{prop:decomprim} and the results in \cite{Durand:2012} we use here, our approach provides a high complexity. 

Our result is for one-dimensional sequences.
What can be said about multidimensional sequences generated by substitution rules ? 
or self-similar tilings ? 
It seems hopeless to generalize our method to tilings,
although the main and key result we use to solve the HD$0$L ultimate periodicity problem (that is, the main result in \cite{Durand:1998}, see \cite{Durand:2012}) has been generalized to higher dimensions by N. Priebe in \cite{Priebe:2000} (see also \cite{Priebe&Solomyak:2001}).
But observe that in \cite{Leroux:2005} the author gives a polynomial time algorithm to know whether or not a Number Decision Diagram defines a Presburger definable set (see also \cite{Muchnik:2003} where it was first proven but with a much higher complexity). 
From this result and \cite{Cerny&Gruska:1986,Salon:1986,Salon:1987} it is decidable to know whether a multidimensional automatic sequence (or fixed point of a multidimensional "uniform" substitution) has a certain type of periodicity (see \cite{Leroux:2005,Muchnik:2003}). 
From \cite{Durand&Rigo:2013} this type of periodicity is equivalent to a block complexity condition.

\section{Words, morphisms, substitutive and HD$0$L sequences}
\label{section2}

In this section we recall classical definitions and notation.
Observe that the notion of substitution we use below could be slightly different from other definitions in the literature.

\subsection{Words and sequences}
 
An {\it alphabet} $A$ is a finite set of elements called {\it
  letters}.
Its cardinality is $|A|$.
A {\it word} over $A$ is an element of the free monoid
generated by $A$, denoted by $A^*$. 
Let $x = x_0x_1 \cdots x_{n-1}$
(with $x_i\in A$, $0\leq i\leq n-1$) be a word, its {\it length} is
$n$ and is denoted by $|x|$. 
The {\it empty word} is denoted by $\epsilon$, $|\epsilon| = 0$. 
The set of non-empty words over $A$ is denoted by $A^+$. 
The elements of $A^{\mathbb{N}}$ are called {\it sequences}. 
If $x=x_0x_1\cdots$ is a sequence (with $x_i\in A$, $i\in \mathbb{N}$) and $I=[k,l]$ an interval of
$\mathbb{N}$ we set $x_I = x_k x_{k+1}\cdots x_{l}$ and we say that $x_{I}$
is a {\it factor} of $x$.  If $k = 0$, we say that $x_{I}$ is a {\it
 prefix} of $x$. 
The set of factors of length $n$ of $x$ is written
$\mathcal{L}_n(x)$ and the set of factors of $x$, or the {\it language} of $x$,
is denoted by $\mathcal{L}(x)$. 
The {\it occurrences} in $x$ of a word $u$ are the
integers $i$ such that $x_{[i,i + |u| - 1]}= u$. 
If $u$ has an occurrence in $x$, we also say that $u$ {\em appears} in $x$.
When $x$ is a word,
we use the same terminology with similar definitions.

The sequence $x$ is {\it ultimately periodic} if there exist a word
$u$ and a non-empty word $v$ such that $x=uv^{\omega}$, where
$v^{\omega}= vvv\cdots $.
In this case $v$ is called a {\em word period} and $|v|$ is called a {\em length period} of $x$. 
It is {\it periodic} if $u$ is the empty word. 
A word $u$ is {\em recurrent} in $x$ if it appears in $x$ infinitely many times.
The sequence $x$ is {\it uniformly recurrent} if all words in its language appear infinitely many times in $x$ and with bounded gaps. 
 
\subsection{Morphisms and matrices} 
Let $A$ and $B$ be two alphabets. Let $\sigma$ be a {\it morphism} 
from $A^*$ to $B^*$. 
When $\sigma (A) \subset B$, we say $\sigma$ is a {\em coding}. 
We say $\sigma$ is
{\it erasing} if there exists $b\in A$ such that $\sigma (b)$ is the
empty word.  
Such a letter is called {\em erasing letter} (w.r.t. $\sigma$).   
If $\sigma (A)$ is included in $B^+$, it induces by concatenation a map
from $A^{\mathbb{N}}$ to $B^{\mathbb{N}}$. 
This map is also denoted by $\sigma$.
With the morphism $\sigma$ is naturally associated its {\em incidence matrix} $M_{\sigma} =
(m_{i,j})_{i\in B , j \in A }$ where $m_{i,j}$ is the number of
occurrences of $i$ in the word $\sigma(j)$.

Let $\sigma $ be an endomorphism.
We say it is {\em primitive} whenever its incidence matrix is primitive ({\em i.e.}, when it has a power with positive coefficients).  
We denote by $\L (\sigma )$ the set of words having an occurrence in some image of $\sigma^n$ for some $n\in \mathbb{N}$.
We call it the language of $\sigma$.

\subsection{Substitutions and substitutive sequences}

We say that an endomorphism $\sigma : A^* \rightarrow A^{*}$ is {\em prolongable on $a\in A$} if there exists a word $u\in A^+$ such that $\sigma(a)=au$ and, moreover, if $\lim_{n\to+\infty}|\sigma^n(a)|=+\infty$.
Prolongable endomorphisms are called {\em substitutions}. 

We say a letter $b\in A$ is {\em growing} (w.r.t. $\sigma$) if $\lim_{n\to+\infty}|\sigma^n(b)|=+\infty$.
We say $\sigma$ is {\em growing} whenever all letters of $A$ are growing.  

Since for all $n\in\mathbb{N}$, $\sigma^n(a)$ is a prefix of $\sigma^{n+1}(a)$ and because
     $(|\sigma^n(a)|)_n$ tends to infinity with $n$, the
     sequence $(\sigma^n(aaa\cdots ))_{n}$ converges (for the usual
     product topology on $A^\mathbb{N}$) to a sequence denoted by
     $\sigma^\omega(a)$. 
The endomorphism $\sigma$ being continuous for the product topology, $\sigma^\omega(a)$ is a fixed point of $\sigma$: $\sigma (\sigma^\omega(a)) = \sigma^\omega(a)$.
A sequence obtained in
    this way (by iterating a prolongable substitution) is said to be {\em purely substitutive} (w.r.t. $\sigma$). 
If $x\in
    A^\mathbb{N}$ is purely substitutive and $\phi:A^*\to B^*$ is a morphism then the sequence $y=\phi (x)$ is said to be a {\em morphic sequence} (w.r.t. $(\sigma , \phi )$). 
When $\phi $ is a coding, we say $y$ is {\em substitutive} (w.r.t. $(\sigma , \phi$)). 
In these cases, when $\sigma$ is primitive, $y$ is uniformly recurrent (see \cite{Queffelec:1987}).

\subsection{D$0$L and HD$0$L sequences}

A {\em D$0$L system} is a triple $G=(A,\sigma ,u)$ where $A$ is a finite alphabet, $\sigma : A^* \to A^*$ is an endomorphism and $u$ is a word in $A^*$.
An {\em HD$0$L system} is a 5-tuple $G=(A,B,\sigma , \phi ,u)$ where $(A,\sigma ,u)$ is a D$0$L system, $B$ is a finite alphabet and  $\phi : A^* \to B^*$ is a morphism.
If it converges, the limit of $(\phi (\sigma^n (uuu\cdots ))_n$ is called {\em HD$0$L sequence}. 

It is clear that substitutive sequences are HD$0$L sequences.
We will show in the last section that HD$0$L sequences are substitutive sequences.
Nevertheless, as the initial data are not the same, it is not enough to solve the ultimate periodicity problem for substitutive sequences. 
Indeed, if $(\sigma^n (uuu\cdots ))_n$ does not converge it seems difficult to decide whether $(\phi (\sigma^n (uuu\cdots ))_n$ converges. We leave this question as an open problem.

\section{Ultimate periodicity of substitutive sequences}
\label{section3}

In this section we prove the decidability of the HD$0$L ultimate periodicity problem for substitutive sequences.

In the sequel $\sigma : A^*\to A^*$ is a substitution prolongable on $a$, $\phi : A^* \to B^*$ is a morphism, $y=\sigma^\omega  (a)$ and $x=\phi (y)$ is a sequence of  $B^\mathbb{N}$.
We have to find an algorithm deciding whether $x$ is ultimately periodic or not.

\subsection{Primitivity assumption and sub-substitutions}

We recall that the HD$0$L ultimate periodicity problem is already solved in the primitive case.

\begin{theo}
\cite{Durand:2012}
\label{th:decidprim}
The HD0L ultimate periodicity problem is decidable in the context of primitive substitutions.
Moreover, a word period can be explicitly computed.
\end{theo}

\begin{proof}
The first part is Theorem 26 in \cite{Durand:2012}.
The second part can be easily deduced from the proof of this theorem.
\end{proof}

The following lemma shows that it is decidable to check that a nonnegative matrix is primitive.

\begin{lemma}
\label{lemme:horn}
\cite{Horn&Johnson:1990}
The $n\times n$ nonnegative matrix $M$ is primitive if and only if $M^{n^2-2n+2}$ has positive entries.
\end{lemma}

From Lemma \ref{lemme:horn}, Section 4.4 and Section 4.5 in \cite{Lind&Marcus:1995} we deduce following proposition. 

\begin{prop}
\label{prop:decomprim}
Let $M=(m_{i,j})_{i,j\in A}$ be a matrix with non-negative coefficients. 
There exist three positive integers 
$p\not = 0$, $q$, $l$, where $q\leq l-1$, and a partition 
$\{ A_i ; 1\leq i\leq l \}$ of $A$ such that
\begin{equation*}
M^p=\bordermatrix{        & A_1     & A_2         & \cdots & A_q       & A_{q+1} & A_{q+2} & \cdots & A_{l}  \cr
                  A_1     & M_1     & 0           & \cdots & 0         & 0       & 0       & \cdots & 0      \cr
                  A_2     & M_{1,2} & M_2         & \cdots & 0         & 0       & 0       & \cdots & 0      \cr
                  \vdots  & \vdots  & \vdots      & \ddots & \vdots    & \vdots  & \vdots  & \vdots & \vdots \cr
                  A_q     & M_{1,q} & M_{2,q}     & \cdots & M_q       & 0       & 0       & \cdots & 0      \cr
                  A_{q+1} & M_{1,q+1} & M_{2,q+1} & \cdots & M_{q,q+1} & M_{q+1} & 0       & \cdots & 0      \cr
                  A_{q+2} & M_{1,q+2} & M_{2,q+2} & \cdots & M_{q,q+2} & 0       & M_{q+2} & \cdots & 0      \cr
                  \vdots  & \vdots    & \vdots    & \ddots & \vdots    & \vdots  & \vdots  & \ddots & \vdots \cr
                  A_l     & M_{1,l}   & M_{2,l}   & \cdots & M_{q,l}   & 0       & 0       & \cdots & M_l    \cr},
\end{equation*}
where the matrices $M_i$ 
have only positive entries or are equal to zero.
Moreover, the partition and $p$ can be algorithmically computed.
\end{prop}

The next three corollaries are consequences of Proposition \ref{prop:decomprim}.

The following corollary will be helpful to change the representation of $x$ (in terms of $(\sigma , \phi )$) to a more convenient representation.

\begin{coro}
\label{coro:changing}
Let $\tau : A^* \to A^*$ be an endomorphism whose incidence matrix has the form of $M^p$ in Proposition \ref{prop:decomprim}.
Then, for all $b\in A$ and all $j\geq 1$, the letters having an occurrence in $\left(\tau^{|A |}\right)^j (b)$ or $\left(\tau^{|A |}\right)^{j+1} (b)$ are the same. 
\end{coro}

\begin{proof}
Let us take the notation of Proposition \ref{prop:decomprim}.
Let $A^{(0)}$ (resp. $A^{(1)}$) be the set of letters $b$ belonging to some $A_i$ where $M_i$ is the null matrix (resp. is not the null matrix).

The conclusion is a consequence of the following two remarks.
Let $b\in A_i$.
From the shape of the incidence matrix of $\tau $ we get:

\begin{itemize}
\item
If $b$ belongs to $A^{(1)}$, 
then all letters occurring in $\tau (b)$ occur in $\tau^n (b)$ for all $n$.
\item
If $b$ belongs to $A^{(0)}$, 
then all letters occurring in $\tau (b)$ belong to some $A_j$ with $j>i$.
\end{itemize}
This achieves the proof.
\end{proof}

\begin{coro}
\label{coro:deciderase}
It is decidable whether a given letter is growing for a given endomorphism.
\end{coro}

\begin{proof}
Let $\tau$ be an endomorphism.
Let us take the notation of Proposition \ref{prop:decomprim}.
Let $A^{(0)}$ (resp. $A^{(1)}$) be the set of letters $b$ belonging to some $A_i$ where $M_i$ is the null matrix (resp. the $1\times 1$ matrix $[1]$).
The letters belonging to $A\setminus (A^{(0)} \cup A^{(1)})$ are growing.

Let $b\in A_i \cap A^{(1)}$ for some $i$.
Then, from Corollary \ref{coro:changing}, $b$ is non-growing (w.r.t. $\tau$) if and only if 
all letters occurring in $\tau^{p|A|} (b)$, except $b$, are erasing with respect to $\tau^{p|A|}$.
Let $A'$ be the set of such non-growing letters.

Let $b\in A_i \cap A^{(0)}$ for some $i$.
Then $b$ is non-growing if and only if all letters occurring in $\tau^{p|A|} (b)$ are erasing with respect to $\tau^{p|A|}$ or belong to $A'$.

Moreover, from Proposition \ref{prop:decomprim} and Corollary \ref{coro:changing} we can decide whether a letter is erasing w.r.t. $\tau^{p|A|}$.
This achieves the proof.
\end{proof}

In what follows we keep the notation of Proposition \ref{prop:decomprim}.
We will say that $\{ A_i; 1\leq i\leq l \}$ is a {\it primitive 
component partition of $A$ (with respect to $M$)}, the $A_i$ being the {\em primitive components}. 
If $i$ belongs to $\{ q+1, \cdots,l \}$ we will say that $A_i$ is a 
{\it principal primitive component of $A$ (with respect to $M$)}.

Let $\tau: A^*\rightarrow A^{*}$ be a substitution whose incidence matrix has the form of $M^p$ in Proposition \ref{prop:decomprim}. 
Let $i\in \{ q+1,\cdots,l \}$. We 
denote $\tau_i$ the restriction $\tau_{/A_i}: A_i^*\rightarrow A^*$ 
of $\tau $ to $A_i^*$. Because $\tau_i (A_i)$ is included in 
$A_i^*$ we can consider that $\tau_i$ is an endomorphism of $A_i^*$ whose incidence matrix is $M_i$. 
When it defines a substitution, we say it is a {\it sub-substitution of $\tau$}. 
Moreover the matrix $M_i$ has positive 
coefficients which implies that the substitution $\tau_i $ 
is primitive.

A non-trivial primitive endomorphism always has some power that is a substitution.
For non-primitive endomorphisms we have the following corollary.

\begin{coro}
\label{coro:subsub}
Let $\tau : A^* \to A^*$ be an endomorphism whose incidence matrix has the form of $M^p$ in Proposition \ref{prop:decomprim}.
Then, with the notation of Proposition \ref{prop:decomprim}, there exists $k\leq |A|^{|A|}$ satisfying : for all $i\geq q+1$ such that $M_i$ is neither a null matrix nor the $1\times 1$ matrix $[1]$, the endomorphism $\tau_i^k$  is a (primitive) substitution for some letter in $A_i$.
\end{coro}

\begin{proof}
We only have to check there exists $k\leq |A|^{|A|}$ such that for all $i\geq q+1$ there exists a letter $b\in A_i$ satisfying $\tau_i^k (b) = bu$ for some non-empty word $u$. 

Let $i\geq q+1$ and $c\in A_i$. 
There exist $k_i\geq 1$ and $j\geq 0$, with $k_i+j \leq |A|$, such that $\tau_i^{j} (c)$ and $\tau_i^{k_i+j} (c)$
start with the same letter $b$.
That is to say, $\tau_i^{k_i} (b) = bu$ for some $u$.
To conclude, it suffices to take $k=k_{q+1} \cdots k_l$.  
\end{proof}

The following lemma is easy to establish.

\begin{lemma}
\label{lemma:subper}
Let $x=\phi (\sigma^\omega (a))$. 
If $x=uv^\omega$, where $v$ is not the empty word, then each sub-substitution $\sigma'$ of $\sigma$ such that $\L (\sigma' )\subset \L (\sigma^\omega (a))$ verifies $\phi (\L(\sigma'))\subset \L(v^\omega )$.    
\end{lemma}

\begin{proof}
Let $\sigma'$ be a sub-substitution of $\sigma $. 
Its incidence matrix being primitive, there exists an uniformly recurrent sequence $z$ such that $\mathcal{L} (\sigma') = \mathcal{L} (z)$ (see \cite{Queffelec:1987}).
Thus, the words of $\mathcal{L} (\sigma')$ appear infinitely many times in $\sigma^\omega (a)$.
Finally, for all $w\in \mathcal{L} (\sigma')$, $\phi (w)$ should occur in $v^\omega$.
\end{proof}

\subsection{Reduction of the problem}

It may happen, as for $\sigma $ defined by $a\mapsto ab$, $b\mapsto a$, $c\mapsto c$, that some letter of the alphabet, here the alphabet is $\{ a,b,c\}$, does not appear in $\sigma^\omega (a)$. 
It is preferable to avoid this situation. 
Corollary \ref{coro:changing} enables us to avoid this algorithmically.
We explain this below.
Indeed, from Proposition \ref{prop:decomprim}, taking a power of $\sigma$ (that can be algorithmically found) if needed, we can suppose 

\begin{itemize}
\item[(P1)]
the incidence matrix of $\sigma$ has the form of $M^p$ in Proposition \ref{prop:decomprim}.
\end{itemize}

Then, consider $\sigma^{|A|}$ instead of $\sigma$.
Hence, $\sigma $ will continue to satisfy (P1) and, from Corollary, \ref{coro:changing} we have 

\begin{itemize}
\item[(P2)]
for all $b\in A$ and all $j\geq 1$ the letters having an occurrence in $\sigma^j (b)$ or $\sigma^{j+1} (b)$ are the same. 
\end{itemize}

Notice that, as $\sigma $ is a substitution, taking a power of $\sigma$ instead of $\sigma$ will change neither $y$ nor $x$. It will not be the case when we will deal with endomorphisms which are not substitutions.

Let $A'$ be the set of letters appearing in $\sigma^\omega (a)$.
From (P2) it can be checked that $\sigma (A')$ is included in $A'^*$ and that the set of letters appearing in $\sigma (a)$ is $A'$.
Thus $\sigma'$, the restriction of $\sigma$ to $A'$, defines a substitution prolongable on $a$ satisfying $\sigma'^\omega (a)=\sigma^\omega (a)$ such that all letters of $A'$ have an occurrence in $\sigma'^\omega (a)$ and all letters of $A'$ occur in $\sigma' (a)$.
Hence we can always suppose $\sigma $ and $a$ satisfy the following condition.

\begin{itemize}
\item[(P3)]
The set of letters occurring in $\sigma^\omega (a)$ is $A$.
\end{itemize}

When we work with morphic sequences it is much simpler to handle with non-erasing substitutions and even better to suppose that $\phi$ is a coding. 
Such a reduction is possible as shown in \cite{Cassaigne&Nicolas:2003}.

\begin{theo}
\label{th:cassaignenicolas}
Let $x$ be a morphic sequence. 
Then, $x$ is substitutive with respect to a non-erasing substitution.
\end{theo}

This result was previously proven in \cite{Cobham:1968} and \cite{Pansiot:1983} (see also \cite{Allouche&Shallit:2003} and \cite{Cassaigne&Nicolas:2003}). 
It was shown that it could be algorithmically done in \cite{Honkala:2009}.
In the sequel we give another algorithm.

The proof of J. Cassaigne and F. Nicolas is short and inspired by \cite{Durand:1998}, in particular its second part which is clearly algorithmic.
Whereas the first part (Lemma 2, Lemma 3 and Lemma 4 of \cite{Cassaigne&Nicolas:2003}) is not because it uses the fact that from any sequence of integers, we can extract a subsequence that is either constant or strictly increasing. 
They use these lemmas to show the key point of their proof~: we can always suppose that $\phi$ and $\sigma$ fulfill the following :

\begin{equation}
\label{eq:cassaignenicolas}
|\phi (\sigma (a))| > |\phi (a) | > 0 \hbox{ and } |\phi (\sigma (b))| \geq |\phi (b) | \hbox{ for all } b\in A .
\end{equation}

Below we show that this can be algorithmically realized. 
This provides another algorithm for Theorem \ref{th:cassaignenicolas}.

First let us show that $\sigma $ can be supposed to be non-erasing.
As we explained before, there is no restriction to suppose $\sigma $ satisfies (P1), (P2) and (P3).

As $\sigma $ satisfies (P2), each letter $e$ is either erasing or, for all $l$, $\sigma^l (e)$ is not the empty word. 
Let $A'$ be the set of non-erasing letters and $A''$ the set of erasing letters.
Let $\psi $ be the morphism that sends the elements of $A''$ to the empty word and that is the identity for the other letters. 
Then, we define $\sigma'$ to be the unique endomorphism defined on $A'$ satisfying $\psi \circ \sigma = \sigma' \circ \psi$.
Observe that $\sigma'$ is easily algorithmically definable and prolongable on $a$.
Moreover we have $\sigma \psi = \sigma$. 
Let $z = \sigma'^\omega (a)$.
Then, $\psi (y) = z$ and $\sigma (z) = y$.

Notice that $\sigma'$ is non-erasing. 
Indeed, if $\sigma ' (a') = \epsilon $ for some $a'\in A'$, then $\psi (\sigma (a')) = \sigma' (a') = \epsilon$.
Hence $\sigma (a') = b_1 \cdots b_k$ where the $b_i$'s belong to $A''$.
Then $\sigma^2 (a') = \epsilon$.
But, from Property (P2), $\sigma^2 (a')$ is not the empty word.

Thus we can also consider  

\begin{itemize}
\item
[(P4)]
$\sigma $ is non-erasing.
\end{itemize}

Consequently, from (P2), $|\phi \circ \sigma (\sigma (a))| > |\phi (\sigma (a))| > |\phi (a) |$,
otherwise $\phi (\sigma^\omega (a))$ would not be an infinite sequence.
Hence, replacing $\phi $ with $\phi\circ \sigma $ if needed, we can suppose $\phi $ and $\sigma $ are such that
$|\phi (\sigma (a))| > |\phi (a) | > 0$.

Moreover, we claim that $\sigma^{2}(b) = \sigma (b)$ for all non-growing letters $b\in A$.
Let $b$ be a non-growing letter. 
As $\sigma$ is non-erasing we necessarily have $|\sigma^{2}(b)| \geq |\sigma (b)|$.
Suppose $|\sigma^{2}(b)| > |\sigma (b)|$.
Then, the letters occurring in $\sigma^{2} (b)$ and $\sigma (b)$ being the same, we would have $|\sigma^n (b)|\geq n+1$ for all $n$, and, $b$ would not be growing.
Consequently, $|\sigma^{2}(b)| = |\sigma (b)|$.
Let $\sigma (b) = b_1b_2 \cdots b_l$.
Then, $|\sigma (b_i) | = 1$ for all $i$, and, from the shape of the incidence matrix of $\sigma$, $\sigma (b_i) = b_i$ for all $i$.

Therefore, replacing $\phi$ with $\phi \circ \sigma$ if needed, we can suppose $|\phi (\sigma^n (b))| \geq  |\phi (b)| $ for all non-growing letter $b$ and all $n$.

Again, replacing $\sigma $ with $\sigma^{k}$, where $k=\max_{a\in A} |\phi (a)|$, if needed, we can suppose \eqref{eq:cassaignenicolas} holds for $\sigma$ and $\phi$.

Hence, together with the argument of the proof of Theorem \ref{th:cassaignenicolas} we obtain the algorithm we are looking for.
This is summarized in the following theorem (first proved in \cite{Honkala:2009}).

\begin{theo}
There exists an algorithm that given $\phi$ and $\sigma $ compute a coding $\varphi$ and a non-erasing substitution $\tau$, prolongable on $a$, such that $x=\varphi (z)$ where $z=\tau^\omega (a)$.
\end{theo}

Thus, in the sequel we suppose $\phi $ is a coding and $\sigma$ is a non-erasing substitution.
We end this section with a technical lemma checking the ultimate periodicity.

\begin{lemma}
\label{lemma:finalcheck}
Let $t\in A^\mathbb{N}$, $\varphi $ be a coding defined on $A^*$, $z=\varphi (t)$, and,  $u$ and $v$ be non-empty words.
Then, $z=uv^\omega$ iff and only if for all recurrent words $B = b_1 b_2 \cdots b_{2|v|}\in \L (t)$, where the $b_i$'s are letters, there exist $r_B\in \{ 0,1,2 \}$, $s_B$ and $p_B$ such that

\begin{enumerate}
\item
$\varphi (B) = s_{B}v^{r_{B}}p_{B}$ where $s_{B}$ is a suffix of $v$ and $p_{B}$ a prefix of $v$, and, 
\item
for all recurrent words $BB' \in \L (t)$, where $B'$ is a word of length $2|v|$, $p_{B}s_{B'}$ is equal to $v$ or the empty word.
\end{enumerate}
\end{lemma}

\begin{proof}
The proof is left to the reader.
\end{proof}

\subsection{The case of substitutive sequences with respect to growing substitutions}

In the sequel we suppose $\sigma$ is a growing substitution.
From Corollary \ref{coro:deciderase} it is decidable to know whether we are in this situation.

We recall that from the previous section we can suppose $\phi$ is a coding and that $\sigma $ satisfies 
(P1), (P2), (P3) and (P4).

\begin{lemma}
\label{lemma:langequal}
Let $u$ and $v$ be two words.
It is decidable to check whether or not $\L(u^\omega)$ is equal to $\L (v^\omega)$. 
\end{lemma}

\begin{lemma}
\label{lemma:recletter}
The set of recurrent letters in $\sigma^\omega (a)= c_0 c_1 \cdots $ is algorithmically computable.
Moreover there is a computable $i$ such that all letters occurring in $c_i c_{i+1} \cdots $ are recurrent. 
\end{lemma}

\begin{proof}
Let $\sigma (a) = au$.
Then, $\sigma^\omega (a) = a u \sigma (u) \sigma^2 (u) \cdots $.
Thus, from (P2), a letter is recurrent if and only if it appears in $\sigma (u)$.
Moreover, all letters occurring in $\sigma (u) \sigma^2 (u) \cdots $ are recurrent.
\end{proof}

\begin{lemma}
\label{lemma:suboflengthn}
The set of recurrent words of length $n$ in $\sigma^\omega (a) = c_0c_1\cdots $ is algorithmically computable. 
Moreover there is a computable $i$ such that all words of length $n$ occurring in $c_i c_{i+1} \cdots $ are recurrent. 

\end{lemma}

\begin{proof}
Let $n\in \mathbb{N}$.
Let $w_0$ be the prefix of length $n$ of $\sigma^n (a)$. 
Let $w_1,\dots , w_{j_1}$ be the words of length $n$ appearing in $\sigma (w_0)$. 
Then we do the same for $w_1$.
We obtain some new words of length $n$: $w_{j_1+1}, \dots , w_{j_2}$.
We proceed similarly with $w_2, w_3$ and so on, until all the $w_i$ are handled and no new words appear. 
At this point, the set $A'$ of all collected words is the set of all words of length $n$ occurring in $\sigma^\omega (a)$. 

It remains to find the words in $A'$ that are recurrent in $\sigma^\omega (a)$.

Consider $A'$ as a new alphabet and $\sigma_n : A'^* \rightarrow A'^*$
the endomorphism defined, for all $(a_1\cdots
  a_n)$ in $A'$, by
$$
\sigma_n ((a_1\cdots a_n)) = (b_1\cdots b_n)(b_2\cdots
b_{n+1})\cdots (b_{|\sigma (a_1)|}\cdots b_{|\sigma (a_1)|+n-1})
$$
where $\sigma(a_1\cdots a_n) = b_1\cdots b_k$. 
Let $\sigma^\omega (a) = c_0 c_1 \cdots$, with $c_i \in A$, $i\geq 0$.
It is easy to check that $\sigma_n $ is prolongable on $c=(c_0 c_1 \cdots c_{n-1})$ and that 

$$
\sigma_n^\omega (c) = (c_0\cdots c_{n-1})(c_1\cdots c_{n})(c_2\cdots c_{n+1}) \cdots .
$$

For details, see Section V.4 in \cite{Queffelec:1987}.
Thus a word $w$ of length $n$ is recurrent in $\sigma^\omega (a)$ if and only if $(w)$ (which is a letter of $A'$) is recurrent in $\sigma_n^\omega (c)$.

We achieve the proof using Lemma \ref{lemma:recletter}
\end{proof}

\begin{theo}
\label{theo:decidgrowing}
The HD0L ultimate periodicity problem is decidable for substitutive sequences w.r.t. growing substitutions.
Moreover, some $u$ and $v$ in the description of the problem can be computed.
\end{theo}

\begin{proof}
In this proof we suppose $\sigma$ is growing.
Let us use the notation of Proposition \ref{prop:decomprim}.
From Corollary \ref{coro:subsub}, taking a power of $\sigma$ (less than $|A|^{|A|}$) if needed, we can suppose that for all $i\geq q+1$ the endomorphism $\sigma_i : A_i^* \to A_i^*$  defines a primitive sub-substitution w.r.t. some letter $a_i\in A_i$.
We recall that all sub-substitutions are primitive.
We notice that, in the growing case, there is at least one sub-substitution.

Observe that for all $i\geq q+1$ and $b\in A_i$, the word $\sigma^n (b)=\sigma_i^n (b)$ is recurrent in $\sigma^\omega (a)$.
Thus, to check the periodicity of $x$, we start checking with Theorem \ref{th:decidprim} that, for all $i\geq q+1$, the sequence $\phi (\sigma_i^\omega (a_i ))$ is periodic.
We point out that when the language is periodic then a word period $w(\sigma_i)$ can be computed.
If for some $\sigma_i$ the sequence $\phi (\sigma_i^\omega (a_i ))$ is not periodic then $x$ cannot be ultimately periodic.
Indeed, suppose $x = uv^\omega$.
As longer and longer words occurring in $\phi (\sigma_i^\omega (a_i ))$ occurs in $x$, the uniform recurrence would imply that $\phi (\sigma_i^\omega (a_i )) = v^\omega $.

Then, we check that all the languages $\L (w(\sigma_i)^\omega  )$ are equal using Lemma \ref{lemma:langequal}. 
From Lemma \ref{lemma:subper}, if this checking fails, then $x$ is not periodic.

Hence we suppose it is the case : There exists a word $v$ that is algorithmically given by Theorem \ref{th:decidprim} such that $\phi ( (w(\sigma_i))^\omega ) = \L(v^\omega )$ for all $i$.
Consequently, we should check whether there exists $u$ such that $x=uv^\omega$.

We conclude using Lemma \ref{lemma:suboflengthn} and Lemma \ref{lemma:finalcheck}.
\end{proof}

\subsection{The case of substitutive sequences with respect to non-growing substitutions}
\label{subsec:nonnonnon}

In the sequel we suppose that $\sigma$ is a non-growing substitution.
From Corollary \ref{coro:deciderase} it is decidable to know whether we are in this situation.
We recall that from the previous section we can suppose $\phi$ is a coding and that $\sigma $ satisfies 
(P1), (P2), (P3) and (P4).

\begin{lemma}
\label{lemma:pansiot}
\cite[Th\'eor\`eme 4.1]{Pansiot:1984}
The substitution $\sigma$ satisfies exactly one one the following two statements.
\begin{enumerate}
\item
\label{lemma:enum:pansiot2}
The length of words (occurring in $\sigma^\omega (a)$) consisting of non-growing letters is bounded.
\item
\label{lemma:enum:pansiot3}
There exists a growing letter $b\in A$, occurring in $\sigma^\omega (a)$, such that $\sigma (b) =  vbu$ (or $ubv$) with $u\in C^*\setminus \{ \epsilon \}$ where $C$ is the set of non-growing letters.
\end{enumerate}
Moreover, in the situation \eqref{lemma:enum:pansiot2} the sequence $\sigma^\omega (a)$ can be algorithmically defined as a  substitutive sequence w.r.t. a growing substitution.
\end{lemma}

\begin{lemma}
\label{lemma:decidpansiot}
It is decidable to know whether $\sigma$ satisfies \eqref{lemma:enum:pansiot2} or \eqref{lemma:enum:pansiot3} of Lemma \ref{lemma:pansiot}.
\end{lemma}

\begin{proof}
It can be easily algorithmically checked whether we are in the situation \eqref{lemma:enum:pansiot3} of Lemma \ref{lemma:pansiot}. 
Thus it is decidable to know whether we are in situation \eqref{lemma:enum:pansiot2} of Lemma \ref{lemma:pansiot}. 
\end{proof}

\begin{theo}
The HD0L ultimate periodicity problem is decidable for substitutive sequences w.r.t. non-erasing substitutions.
Moreover, some $u$ and $v$ in the description of the problem can be computed.
\end{theo}

\begin{proof}
From Theorem \ref{theo:decidgrowing}, Lemma \ref{lemma:pansiot} and Lemma \ref{lemma:decidpansiot} it remains to consider that $\sigma $ satisfies \eqref{lemma:enum:pansiot3} in Lemma \ref{lemma:pansiot} : Let $b$ be a letter occurring in $\sigma^\omega (a)$ such that $\sigma (b) =  vbu$ (or $ubv$) with $u\in C^*\setminus \{ \epsilon \}$ where $C$ is the set of non-growing letters. 
Then, for all $n$, $\sigma^{n+1} (b) = \sigma^n (v)bu \sigma (u) \cdots \sigma^{n}(u)$.
As the sequence $(|\sigma^n (u)|)_n$ is bounded, there exist $i$ and $j$, $i<j$, such that $\sigma^i (u) = \sigma^j (u)$.
Let $u' = \sigma^i (u) \sigma^{i+1} (u) \cdots  \sigma^{j-1} (u)$.
Then, we get $\L (u'^\omega ) \subset \L (\sigma )$.
We conclude using Lemma \ref{lemma:suboflengthn} and Lemma \ref{lemma:finalcheck}.
\end{proof}

This ends the proof of Theorem \ref{theo:main} for substitutive sequences.

\begin{theo}
\label{theo:perforsub}
Suppose the sequence $x$ is substitutive with respect to $(\sigma ,\phi)$.
Then, it is decidable whether $x$ is ultimately periodic: $x=uv^\omega$ for some $u$ and non-empty $v$.
Moreover, we can compute such $u$ and $v$. 
\end{theo}

\section{Ultimate periodicity of HD$0$L sequences}
\label{section4}

In this section we end the proof of the Theorem \ref{theo:main}: We solve the HD$0$L ultimate periodicity problem.
We use the notation introduced in the input of the problem.
We recall that in the previous section we prove this theorem for a special case of HD$0$L sequences: the substitutive sequences. 
These sequences are very convenient as, by definition, there is no problem with the existence of the limit in the statement of the HD$0$L ultimate periodicity problem.
We gave, in Section \ref{subsec:organization}, an example of an HD$0$L sequence where the sequence $(\sigma^n (a))_n$  does not converge but $(\phi (\sigma^n (a)))_n$ does.

Thus, in the general case, it would be convenient (but not necessary) to be able to decide the existence of the limit.
As we did not succeed to solve this decidability problem, we leave this question as an open problem. 
We proceed in a different way.

Let us consider the input of the HD$0$L ultimate periodicity problem.

\begin{lemma}
\label{lemma:tendsto}
Let $a\in A$.
Suppose $\sigma$ satisfies (P1) and (P2).
Then, it is decidable whether:

\begin{enumerate}
\item
\label{lemma:tendstozero} 
$(|\phi (\sigma^n (a))|)_n$ tends to $0$,
\item
\label{lemma:phisigmainfty}
$(|\phi (\sigma^n (a))|)_n$ tends to infinity.
\end{enumerate}
Moreover, if $(|\phi (\sigma^n (a))|)_n$ does not tend to infinity then it is bounded.
\end{lemma}

\begin{proof}
Let $A'$ be the set of letters occurring in $\sigma(a)$.
We prove the decidability of \eqref{lemma:tendstozero}.
From (P2), for all $n\geq 1$, the set of letters occurring in $\sigma^n (a)$ is $A'$.
Then, $(|\phi (\sigma^n (a))|)_n$ tends to $0$ if and only if $\phi (a')$ is the empty word for all $a'\in A'$.

We prove the decidability of \eqref{lemma:phisigmainfty}.
Let us consider the notation of Proposition \ref{prop:decomprim} for $\sigma$.
As $\sigma$ satisfies (P1) we can suppose $p=1$.

Suppose $a$ belongs to $A_l$.
Then $(|\phi (\sigma^n (a))|)_n$ tends to infinity if and only if $M_l$ is neither the $1\times 1$-matrix $[1]$ nor the null matrix, and, there exists a letter $b\in A_l$ such that $\phi (b)$ is not the empty word.
Thus for such a letter the problem is decidable.
Moreover, if $(|\phi (\sigma^n (a))|)_n$ does not tend to infinity, then it is bounded.

Now we proceed by a finite induction. 
Suppose the problem is decidable for all letters in $\cup_{n+1\leq j\leq l}A_j$. 
We show it is decidable for all letters in $\cup_{n\leq j\leq l}A_j$.
 
Suppose $a$ belongs to $A_n$.
If $M_n$ is the null matrix, then we conclude with our induction hypothesis.

Suppose $M_n $ is the $1\times 1$-matrix $[1]$.
Then, $(|\phi (\sigma^n (a))|)_n$ tends to infinity if and only if there is a letter $a'$ in $A'\setminus \{ a\}$ such that $(|\phi (\sigma^n (a'))|)_n$ does not tend to zero.
Hence the decidability is deduced from \eqref{lemma:tendstozero}.
Moreover, if $(|\phi (\sigma^n (a))|)_n$ does not tend to infinity, then it is bounded.

Suppose $M_n $ is neither the $1\times 1$-matrix $[1]$ nor the null matrix.
Then, $(|\phi (\sigma^n (a))|)_n$ tends to infinity if and only if there exists a letter in $A'$ such that $\phi (a')$ is not empty.
Moreover, if $(|\phi (\sigma^n (a))|)_n$ does not tend to infinity, then it goes to $0$ and thus is bounded.
\end{proof}

Let us conclude with the HD$0$L ultimate periodicity problem.

Let us first suppose that $\sigma$ satisfies (P1) and (P2).

Let $w = w_0\cdots w_{|w|-1}$ where the $w_i$'s belong to $A$.
As we want to test the ultimate periodicity, from Lemma \ref{lemma:tendsto}, we can suppose $(|\phi (\sigma^n (w_0))|)_n$ tends to infinity.
Consequently we can suppose $w=w_0$.
We set $a=w_0$.

Let $j_0$ be the smallest integer less or equal to $|A|+1$ such that $\sigma^{j_0} (a)$ and  $\sigma^{j_0+n_0} (a)$ start with the same first letter for some $n_0$ verifying $j_0+n_0\leq |A|+1$. 
Such integers exist from the pigeon hole principle.  
We also assume $n_0$ is the smallest such integer.
Let $a_i$ be the first letter of $\sigma^{j_0+i} (a)$, $0\leq i\leq n_0-1$.
Notice that if $(|\phi (\sigma^{j_0+i +kn_0} (a_i))|)_k $ tends to infinity then $(\phi (\sigma^{j_0+i +kn_0} (a_i)))_k$ converges in $B^\mathbb{N}$.
From Lemma \ref{lemma:tendsto} it is decidable to know whether $(|\phi (\sigma^{j_0+i +kn_0} (a_i))|)_k $ tends to infinity. 
Let $\Lambda$ be the set of such $a_i$'s.
Then, the set of accumulation points in $B^\mathbb{N}$ of $(\phi (\sigma^n (w)))_n$ is computable: it is the set of the infinite sequences $\lim_{k\to +\infty } \phi (\sigma^{j_0+i +kn_0} (a_i))$ where $a_i$ belongs to $\Lambda$.

Consequently, $(\phi (\sigma^n (w)))_n$ converges to an ultimately periodic sequence if and only if there exist $u,v\in B^*$ such that for all $0\leq i\leq n_0-1$, $\lim_{k\to +\infty } \phi (\sigma^{j_0+i +kn_0} (a_i))=uv^\omega$.
Thus to decide whether $(\phi (\sigma^n (w)))_n$ converges to an ultimately periodic sequence, we first have to check (using Theorem \ref{theo:perforsub}) that for all $0\leq i\leq n_0-1$, $\lim_{k\to +\infty } \phi (\sigma^{j_0+i +kn_0} (a_i))=u_iv_i^\omega$, for some computable  $u_i,v_i \in B^*$.
Then, we check whether the sequences $u_iv_i^\omega$ are equal (which can be algorithmically realized).

\medskip

Let then $\sigma$ be an arbitrary morphism. 
From Proposition \ref{prop:decomprim} we can suppose that $\sigma^p$ satisfies (P1) and (P2) for some computable $p>0$.
Then, we proceed as before for the couples $(\sigma^p , \phi \circ \sigma^i )$, $0\leq i\leq p-1$: We test their ultimate periodicity and then we compare the results to finally decide.

\medskip

{\bf Acknowledgements.}
The author thanks the referees for their valuable and relevant comments.
He also thanks the ANR program SubTile for its financial support.

\bibliographystyle{new2}
\bibliography{periodicite}

\def\ocirc#1{\ifmmode\setbox0=\hbox{$#1$}\dimen0=\ht0 \advance\dimen0
  by1pt\rlap{\hbox to\wd0{\hss\raise\dimen0
  \hbox{\hskip.2em$\scriptscriptstyle\circ$}\hss}}#1\else {\accent"17 #1}\fi}
\begin{thebibliography}{10}

\bibitem[Allouche, Rampersad and Shallit
  2009]{Allouche&Rampersad&Shallit:2009}\ \\
J.-P. Allouche, N.~Rampersad, and J.~Shallit.
\newblock Periodicity, repetitions, and orbits of an automatic sequence.
\newblock {\em Theoret. Comput. Sci.} {\bf 410}(30-32) (2009), 2795--2803.

\bibitem[Allouche and Shallit 2003]{Allouche&Shallit:2003}\ \\
J.-P. Allouche and J.~O. Shallit.
\newblock {\em Automatic Sequences, Theory, Applications, Generalizations}.
\newblock Cambridge University Press, 2003.

\bibitem[Bell, Charlier, Fraenkel, and Rigo
  2009]{Bell&Charlier&Fraenkel&Rigo:2009}\ \\
J.~P. Bell, E.~Charlier, A.~S. Fraenkel, and M.~Rigo.
\newblock A decision problem for ultimately periodic sets in non-standard
  numeration systems.
\newblock {\em Internat. J. Algebra Comput.} {\bf 9} (2009), 809--839.

\bibitem[Cassaigne and Nicolas 2003]{Cassaigne&Nicolas:2003}\ \\
J.~Cassaigne and F.~Nicolas.
\newblock Quelques propri\'et\'es des mots substitutifs.
\newblock {\em Bull. Belg. Math. Soc. Simon Stevin} {\bf 10} (2003), 661--676.

\bibitem[Cern\'y and Gruska 1986a]{Cerny&Gruska:1986}\ \\
A.~Cern\'y and J.~Gruska.
\newblock Modular trellises.
\newblock In G.~Rozenberg and A.~Salomaa, editors, {\em The Book of {L}}, pp.
  45--61. Springer-Verlag, 1986.

\bibitem[Cobham 1968]{Cobham:1968}\ \\
A.~Cobham.
\newblock On the {Hartmanis-Stearns} problem for a class of tag machines.
\newblock In {\em IEEE Conference Record of 1968 Ninth Annual Symposium on
  Switching and Automata Theory}, pp. 51--60, 1968.
\newblock Also appeared as IBM Research Technical Report RC-2178, August 23
  1968.

\bibitem[Durand 1998]{Durand:1998}\ \\
F.~Durand.
\newblock A characterization of substitutive sequences using return words.
\newblock {\em Discrete Math.} {\bf 179} (1998), 89--101.

\bibitem[Durand 2012]{Durand:2012}\ \\
F.~Durand.
\newblock {HD}$0${L} $\omega$-equivalence and periodicity problems in the
  primitive case ({T}o the memory of {G}. {R}auzy).
\newblock {\em J. Unif. Distrib. Theory} {\bf 7} (2012), 199--215.

\bibitem[Durand and Rigo]{Durand&Rigo:2013}\ \\
F.~Durand and M.~Rigo.
\newblock Multidimensional extension of the {M}orse-{H}edlund theorem.
\newblock {\em Eur. J. Comb.} {\bf 34} (2013), 391--409.

\bibitem[Harju and Linna 1986]{Harju&Linna:1986}\ \\
T.~Harju and M.~Linna.
\newblock On the periodicity of morphisms on free monoids.
\newblock {\em RAIRO Inform. Th{\'e}or. App.} {\bf 20} (1986), 47--54.

\bibitem[Honkala 1986]{Honkala:1986}\ \\
J.~Honkala.
\newblock A decision method for the recognizability of sets defined by number
  systems.
\newblock {\em RAIRO Inform. Th{\'e}or. App.} {\bf 20} (1986), 395--403.

\bibitem[Honkala 2008]{Honkala:2008}\ \\
J.~Honkala.
\newblock Cancellation and periodicity properties of iterated morphisms.
\newblock {\em Theoret. Comput. Sci.} {\bf 391} (2008), 61--64.

\bibitem[Honkala 2009]{Honkala:2009}\ \\
J.~Honkala.
\newblock On the simplification of infinite morphic words.
\newblock {\em Theoret. Comput. Sci.} {\bf 410} (2009), 997--1000.

\bibitem[Honkala and Rigo 2004]{Honkala&Rigo:2004}\ \\
J.~Honkala and M.~Rigo.
\newblock Decidability questions related to abstract numeration systems.
\newblock {\em Discrete Math.} {\bf 285} (2004), 329--333.

\bibitem[Horn and Johnson 1990]{Horn&Johnson:1990}\ \\
R.~A. Horn and C.~R. Johnson.
\newblock {\em Matrix analysis}.
\newblock Cambridge University Press, 1990.

\bibitem[Lecomte and Rigo 2010]{Lecomte&Rigo:2010}\ \\
P.~Lecomte and M.~Rigo.
\newblock Abstract numeration systems.
\newblock In {\em Combinatorics, automata and number theory}, Vol. 135 of {\em
  Encyclopedia Math. Appl.}, pp. 108--162. Cambridge Univ. Press, 2010.

\bibitem[Leroux 2005]{Leroux:2005}\ \\
J.~Leroux.
\newblock A polynomial time presburger criterion and synthesis for number
  decision diagrams.
\newblock In {\em 20th IEEE Symposium on Logic In Computer Science (LICS
  2005)}, IEEE Computer Society, pp. 147--156, 2005.

\bibitem[Lind and Marcus 1995]{Lind&Marcus:1995}\ \\
D.~Lind and B.~Marcus.
\newblock {\em An Introduction to Symbolic Dynamics and Coding}.
\newblock Cambridge University Press, 1995.

\bibitem[Maes and Rigo 2002]{Maes&Rigo:2002}\ \\
A.~Maes and M.~Rigo.
\newblock More on generalized automatic sequences.
\newblock {\em J. Autom. Lang. Comb.} {\bf 7} (2002), 351--376.

\bibitem[Mitrofanov preprint 2011]{Mitrofanov:preprint2011}\ \\
I.~Mitrofanov.
\newblock A proof for the decidability of {HD}0{L} ultimate periodicity.
\newblock {\em arXiv:1110.4780}  (2011).

\bibitem[Muchnik 2003]{Muchnik:2003}\ \\
A.~Muchnik.
\newblock The definable criterion for definability in {P}resburger arithmetic
  and its applications.
\newblock {\em Theoret. Comput. Sci.} {\bf 290}(3) (2003), 1433--1444.

\bibitem[Pansiot 1983]{Pansiot:1983}\ \\
J.-J. Pansiot.
\newblock {Hi\'erarchie} et fermeture de certaines classes de {tag-syst\`emes}.
\newblock {\em Acta Informatica} {\bf 20} (1983), 179--196.

\bibitem[Pansiot 1984]{Pansiot:1984}\ \\
J.-J. Pansiot.
\newblock {Complexit\'e} des facteurs des mots infinis {engendr\'es} par
  morphismes {it\'er\'es}.
\newblock In J.~Paredaens, editor, {\em ICALP84}, Vol. 172 of {\em Lecture
  Notes in Computer Science}, pp. 380--389. Springer-Verlag, 1984.

\bibitem[Pansiot 1986]{Pansiot:1986}\ \\
J.-J. Pansiot.
\newblock Decidability of periodicity for infinite words.
\newblock {\em RAIRO Inform. Th{\'e}or. App.} {\bf 20} (1986), 43--46.

\bibitem[Priebe 2000]{Priebe:2000}\ \\
N.~Priebe.
\newblock Towards a characterization of self-similar tilings in terms of
  derived {V}orono\u\i\ tessellations.
\newblock {\em Geom. Dedicata} {\bf 79} (2000), 239--265.

\bibitem[Priebe and Solomyak 2001]{Priebe&Solomyak:2001}\ \\
N.~Priebe and B.~Solomyak.
\newblock Characterization of planar pseudo-self-similar tilings.
\newblock {\em Discrete Comput. Geom.} {\bf 26} (2001), 289--306.

\bibitem[Queff\`elec 1987]{Queffelec:1987}\ \\
M.~Queff{\'e}lec.
\newblock {\em Substitution dynamical systems---spectral analysis}, Vol. 1294
  of {\em Lecture Notes in Mathematics}.
\newblock Springer-Verlag, 1987.

\bibitem[Salon 1986]{Salon:1986}\ \\
O.~Salon.
\newblock Suites automatiques {\`a} multi-indices.
\newblock In {\em {S\'eminaire} de {Th\'eorie} des Nombres de Bordeaux}, pp.
  4.01--4.27, 1986-1987.

\bibitem[Salon 1987]{Salon:1987}\ \\
O.~Salon.
\newblock Suites automatiques {\`a} multi-indices et {alg\'ebricit\'e}.
\newblock {\em C. R. Acad. Sci. Paris} {\bf 305} (1987), 501--504.

\end{thebibliography}

\end{document}